\documentclass{amsart}
\usepackage{amssymb,color,xcolor}
\usepackage{amscd}
\usepackage{enumerate}
\usepackage{siunitx}
\usepackage{bm}
\usepackage{ulem}
\numberwithin{equation}{section}

\usepackage[pagebackref]{hyperref}

\usepackage{mathtools}
\usepackage[tableposition=top]{caption}
\usepackage{booktabs,dcolumn}

\DeclareFontFamily{OT1}{rsfs}{}
\DeclareFontShape{OT1}{rsfs}{n}{it}{<-> rsfs10}{}
\DeclareMathAlphabet{\mathscr}{OT1}{rsfs}{n}{it}

\theoremstyle{plain}

\newtheorem{theorem}{Theorem}[section]
\newtheorem{prop}[theorem]{Proposition}
\newtheorem{lemma}[theorem]{Lemma}

\theoremstyle{definition}

\newtheorem{remark}[theorem]{Remark}

\newcommand{\bal}{\[\begin{aligned}}
\newcommand{\eal}{\end{aligned}\]}

\newcommand{\beq}{\begin{equation}}
\newcommand{\eeq}{\end{equation}}

\newcommand{\SN}{{\mathbb{N}}}

\newcommand{\beeq}{\begin{equation}}\newcommand{\eneq}{\end{equation}}
\def\<{\langle}             \def\>{\rangle}
\newcommand{\al}{\alpha}    \newcommand{\be}{\beta}

    \newcommand{\la}{\lambda}

\newcommand{\ga}{\Gamma}   
\newcommand{\R}{\mathbb{R}}
\newcommand{\N}{\mathbb{N}}
\newcommand{\CC}{\mathbb{C}}

\newcommand{\pa}{\partial}
\newcommand{\les}{{\lesssim}}

\newcommand{\hn}{\mathbb{H}^n}
\newcommand{\e}{\eta}
\newcommand{\De}{\Delta}

\newcommand{\shr}{\frac{1}{\sinh{r}}}

\newcommand{\shu}{\frac{1}{\sinh{u}}}
\newcommand{\hf}{\frac{1}{2}}
\newcommand{\f}{\mathscr{F}}

\newcommand{\uf}{\frac{1}{u}}

\newcommand{\hz}{\tilde{h_z}}
\newcommand{\hzz}{\tilde{\tilde{h_z}}}

\newcommand{\one}{\uppercase\expandafter{\romannumeral1}}
\newcommand{\two}{\uppercase\expandafter{\romannumeral2}}
\newcommand{\three}{\uppercase\expandafter{\romannumeral3}}
\newcommand{\four}{\uppercase\expandafter{\romannumeral4}}
\newcommand{\five}{\uppercase\expandafter{\romannumeral5}}
\newcommand{\six}{\uppercase\expandafter{\romannumeral6}}
\newcommand{\sev}{\uppercase\expandafter{\romannumeral7}}
\newcommand{\eig}{\uppercase\expandafter{\romannumeral8}}

\title
{An alternative Proof of Tataru's dispersive estimates}

\author{Chengbo Wang}
\address{School of Mathematical Sciences\\ Zhejiang University\\ Hangzhou 310027,P.R.China}
\email{wangcbo@zju.edu.cn}
\urladdr{http://www.math.zju.edu.cn/wang}

\author{Xiaoran Zhang$^{*}$}\thanks{* Corresponding author}
\address{School of Mathematical Sciences\\ Zhejiang University\\ Hangzhou 310027,P.R.China}
\email{1025391337@qq.com}

\thanks{
The authors were supported  by NSFC 
11971428 and  NSFC 12141102.}

\date{\today}

\begin{document}

\begin{abstract}
The aim of this article is to give an alternative proof of  Tataru's dispersive estimates for wave equations posed on the hyperbolic space. Based on the formula for the wave kernel on $\hn$, 
we give the proof from the perspective of Bessel potentials, by
exploiting various facts about Gamma functions, modified Bessel functions, and Bessel potentials.
 This leads to our proof being more self-contained than that in Tataru \cite{MR1804518}.
\end{abstract}



\maketitle

\section{Introduction}

In the seminal paper \cite{MR1804518}, Tataru gave an alternative proof of  the Georgiev-Lindblad-Sogge theorem
\cite{GLS97}, concerning on the global existence part of the Strauss conjecture, for nonlinear wave equations posed on $\R^n\times \R$.
By using the relation between wave equations on hyperbolic space $\hn\times\R$ and wave equations on $\R^n \times \R$, he reduced the proof of the weighted Strichartz estimates on $\R^n$ to the dispersive estimates for wave equations on $\hn$.

Let $n\ge 2$, $\rho=(n-1)/2$,
considering wave equations on the hyperbolic space $\hn$, 
$$(\partial_t^2-\De_{\hn}-\rho^2)w=0,\ w(0,x)=w_0(x),\ w_t(0,x)=w_1(x),$$
the solutions can be expressed by $w(t,x)=S(t)w_1(x)+C(t)w_0(x)$ with
\begin{align*}
S(\tau)=\frac{\sin D_0\tau}{D_0},\ C(\tau)=\cos D_0\tau,
\end{align*}
where $D_0=\sqrt{-\Delta_{\hn}-\rho^2}$.
Let $\be>\rho$ and $D=\sqrt{D_0^2+\be^2}$, then
the Tataru's dispersive estimates \cite{MR1804518}
state that
\beeq
\label{Eq:pro1} \|S(\tau)f\|_{L^q}\lesssim \frac{(1+\tau)^{\frac{2}{q}}}{(\sinh \tau)^{(n-1)(\hf-\frac{1}{q})}} \|D^{(n+1)(\hf-\frac{1}{q})-1} f\|_{L^{q^{\prime}}},\  2\leq q<\infty,\eneq
\beeq\label{Eq:pro2} \|C(\tau)f\|_{L^q}\lesssim \frac{1}{(\sinh \tau)^{(n-1)(\hf-\frac{1}{q})}} \|D^{(n+1)(\hf-\frac{1}{q})} f\|_{L^{q^{\prime}}}, 2\leq q<\infty.
\eneq

As is well-known, the dispersive estimates are of fundamental importance in our understanding of the wave equations. For example, it is known from 
Anker-Pierfelice-Vallarino
\cite{MR2902129,MR3345662} that we have Strichartz estimates for a larger range of Lebesgue exponents, and thus
the Strauss conjecture on the hyperbolic space
 for a larger range of nonlinear powers, that is, 
we have small data global existence 
for the Cauchy problem of 
the nonlinear wave equations of the form
$(\partial_t^2-\De_{\hn}-\rho^2)w=\pm |w|^p, \pm |w|^{p-1}w$,
for any sub-conformal and conformal powers ($1<p\le 1+4/(n-1)$).
In the work \cite{MR4026182} of  Sire, Sogge and the first author,
it was shown that Tataru's dispersive estimates could be exploited to give a short alternative proof of the Strauss conjecture on the hyperbolic space.
For recent works on related problems posed on (asymptotically) hyperbolic spaces, we refer the interested readers to 
\cite{MR4169670}
and references therein.

In Tataru's proof, 
in order to use Stein's complex interpolation theorem,
he constructed
the analytic family of operators, with symbol
$$S_{\tau}^{z}(\lambda)=\frac{e^{z^2}}{\ga(z+\rho)} \frac{\sin\lambda\tau}{\lambda}(\lambda^2+\be^2)^{\hf z},\ \be>\rho,\ 
\Re z\in [-\rho, 1]\ 
.$$
Then the kernel $K_{\tau}^z(s)$ is written as
$$
K_{\tau}^z(s)=\int_0^{\infty} S_{\tau}^z(\lambda)\Phi_{\lambda}(s)|c(\lambda)|^{-2}d\lambda,
$$
where $\Phi_{\lambda}(s)$ is from the Fourier transform on spherically symmetric function (“spherical transform”) in $\hn$,
$$
\Phi_{\lambda}(s)=\frac{2^{\rho-1}\ga(\rho+\hf)}{\ga(\rho)(\sinh s)^{2\rho-1}}\int_{-s}^s e^{i\lambda\mu}(\cosh s-\cosh \mu)^{\rho-1}d\mu,
$$
and $c(\lambda)$ is the Harish-Chandra c-function
$$
c(\lambda)=\frac{2^{2\rho-1}\ga(\rho+\hf)\ga(i\lambda)}{\pi^{\hf}\ga(\rho+i\lambda)}.
$$
With the help of these formulas,
 the problem is reduced to controlling $\Phi_{\lambda}(s)$ and $c(\lambda)$ as the main part. On the other hand, we would like to point out that, in the proof, the condition $\be>\rho$ seems to be necessary (for the $L^1\to L^\infty$ estimate) to obtain the appropriate decay of the Fourier transform of $\lambda^{-1}(\lambda^2+\be^2)^{\hf z}|c(\lambda)|^{-2}$, which is holomorphic $S^{\rho-1}$ symbol inside the strip $|\Im\lambda|\leq \rho$ except simple poles at $\pm i\rho$.

In the recent work \cite{MR4026182}, when the spatial dimension is three ($n=3$), Sire, Sogge and the first author gave an alternative proof of Tataru's dispersive estimates with $\be=\rho$, by transferring the logarithmic growth of Bessel potentials to ``$\Im z$". 
Moreover, the proof is slightly more self-contained than the one in  \cite{MR1804518}, since it 
relies only on simple facts about Bessel potentials, and
avoids the heavy spherical analysis on hyperbolic space such as the Harish-Chandra c-function.

As the fundamental solution for the three dimensional wave equation is of the simplest form among all of the spatial dimensions $n\ge 2$, it is interesting to see whether such an elementary alternative proof is possible for other spatial dimensions, which could help 
understanding  the dispersive nature of the wave equations.

In this paper, we make a slight modification of Tataru's argument by introducing the following analytic family of operators
\beeq\label{eq-Stein}
S_z(t) =(z+\rho) e^{z^2}D^z \frac{\sin tD_0}{D_0},\ 
C_z(t) =(z+\rho) e^{z^2}D^{-1+z}\cos tD_0
\ ,\eneq
where $D=\sqrt{D_0^2+\be^2}$.
Then we have the following
\begin{theorem}\label{thm-main}
Let $\be\geq \rho$ for $n=3$ and $\be>\rho$ for $n \neq 3$, there is a constant $C$, independent of $\Im z$, so that for any $z$ with $\Re  z=-\rho$,
\beeq
\label{pro1} \|S_z (t)\|_{L^1(\hn)\to L^{\infty}(\hn)}, \|C_z (t)\|_{L^1(\hn)\to L^{\infty}(\hn)} \leq C(\sinh t)^{-\rho},\ \forall t>0\ .\eneq
\end{theorem}

Notice that,
 by the spectral theorem, we have the following uniform estimates
\beeq\label{pro2} \|S_z (t)\|_{L^2(\hn)\to L^2(\hn)}\leq C(1+t),\|C_z (t)\|_{L^2(\hn)\to L^{2}(\hn)} \leq C\ \forall t>0\ , \eneq
when $\Re z=1$.
By Stein's complex interpolation theorem, \eqref{pro1} and \eqref{pro2} yield  \eqref{Eq:pro1} and \eqref{Eq:pro2}.

Let us conclude this section by discussing a little more about the proof of
Theorem \ref{thm-main}.
Instead of using the spherical transform, we want to start with the functional calculus and wave kernel on $\hn$.
As Bessel potentials arise naturally from the formulas,
we try to give the proof, by exploiting Bessel potentials and its asymptotic behavior. In \S \ref{sec:Bessel}, we introduce Bessel potentials and present several precise estimates of them, which will be crucial in the sequel. In addition, in \S  \ref{sec:wave}, we give some technical treatments on the wave kernel and reduce the proof
of Theorem \ref{thm-main}
to
Proposition \ref{thm-reduce}. Then, in \S \ref{sec:prf}, the alternative proof is obtained with the help of some facts of Bessel potentials and the helpful term ``$(z+\rho)$". Lastly, in the Appendix \S \ref{sec:appen}, we recall some facts of the gamma functions and asymptotic behaviors of modified Bessel functions with imaginary parameter. 

\subsection*{Notation} 
\begin{itemize}
\item $A\lesssim_{y}B$ means that $ A\leq CB$, where $C$ is a constant, which may change from line to line.
To emphasize the dependence of $C$ on certain parameter $y$, we will 
use $A\lesssim_{y} B$.
\item In the multiple integral, $\int_{-1}^1\int_0^{\theta_0}\cdots\int_0^{\theta_{l-1}}$,
we allow the choice of $l=0$ to denote $\int_{-1}^1$, for simplifying notation.
\item $x\wedge y:=\min (x,y)$,
$x\vee y:=\max (x,y)$,
 $\<x\>=\sqrt{1+x^2}$, and
we denote the integer part of $x$ by $[x]$.
\item $\f(f)(\xi)=\int_{-\infty}^{\infty}f(x)e^{-ix\xi}dx$ denotes the Fourier transform .
\end{itemize}

\section{Bessel Potentials and Asymptotic Behaviors}\label{sec:Bessel}
The aim of this section is to give some essential estimates of Bessel potentials. We recall the following well-known formula for Bessel potentials,
with parameter $z\in(-2,-1)$,
\beq\label{defbessel}
\begin{split}
\int_{-\infty}^{\infty} \<\e\>^z e^{-ix\e}d\e
=&\frac{2\pi 2^{(1+z)/2}}{\pi^{\hf}\ga(-z/2)}|x|^{-(1+z)/2}K_{\hf(1+z)}(|x|)\\
=&\frac{2^{2+z/2}\pi^2 e^{-|x|}}{ \ga(-z/2)\ga(1+ z/2)} \int_0^{\infty} e^{-|x|\tau}\left(\hf\tau^2+\tau\right)^{z/2}d\tau,
\end{split}
\eeq
where
$K_\nu$ is the
modified Bessel function of the second kind (\cite[\S 6.3, \S 3.7]{MR1349110})
\beeq\label{eq-Knu}
K_{\nu} (x)=\frac{(\hf \pi )^{\hf} x^{\nu} e^{-x}}{\ga(\nu +\hf)} \int_0^{\infty} e^{-x\tau} (\tau+\hf\tau^2)^{\nu-\hf} d\tau\ ,\ \Re \nu >-\hf, x>0.\eneq
One could directly obtain it by \cite[Chapter \five (26)]{MR0290095} with \cite[Section 6.3]{MR1349110}, or use \cite[(2.10)-(3.6)]{MR143935}.
By analytic continuation, if $x\neq 0$ is fixed, each terms in \eqref{defbessel} are entire functions of $z$, and so the formula is valid for all $z\in\mathbb{C}$, with the fact that $K_{\nu} (|x|)=K_{-\nu} (|x|)$.

By \eqref{defbessel},
we shall denote the Bessel potentials, with parameter $z\in\CC$, by
$F_z(x)$
\begin{equation}\label{eq-Notation-bessel}
F_z(x):=\f (\<\eta\>^{z})(x)=C_z |x|^{-(1+z)/2} K_{\hf(1+z)}(|x|) \ ,\ \forall x\neq 0
\end{equation}
where $C_z=\frac{2\pi2^{ (1+z)/2}}{\sqrt{\pi}\ga(-z/2)}$. With the help of \eqref{gaaa}, it is obvious that
\begin{equation}\label{upperboundofc}
|C_z|\lesssim_{\Re z} \left|\frac{1}{\ga(-\hf z)}\right|\lesssim_{\Re z} e^{ \hf\pi |\Im z|}\ .\eneq
By \eqref{eq-Knu},  a simple computation leads to
\begin{equation}\label{upperboundofk}
|K_{\nu}(x)|\les_{\Re \nu}e^{\pi|\Im\nu|}\times\begin{cases}
x^{-\hf}e^{-x}&x\geq 1\ ,\\
 (\frac{1}{|\Im\nu|}\wedge\ln\frac{1}{x})+1 &  0<x\leq 1,\ \Re\nu=0\ ,\\
x^{-|\Re\nu|}& 0<x\leq 1,\ \Re\nu\neq 0\ ,
\end{cases}
\end{equation}
For the reader's convenience, we give a proof in the appendix \S \ref{sec:appen.2}. Equipped with these estimates, we are able to obtain the following uniform estimates for Bessel potentials

\begin{lemma}\label{lem-2.1}
Let $z\in\CC\backslash \mathbb{R}$, $j\in\mathbb{N}$ and $x\in (0, 1]$, we have the following asymptotic estimates:
\begin{equation}\label{bpdezero}
|F_z^{(j)}(x)|\lesssim_{j,\Re z} e^{2\pi |\Im z|}\times \begin{cases}
x^{-\Re z-j-1}& j> -\Re z-1,\\
(\frac{1}{|\Im z|}\wedge\ln\frac{1}{x})+1 & j=-\Re z-1, \Re z\ \mathrm{odd} ,\\
1 &\mathrm{else}\ .
\end{cases}
\end{equation}
On the other hand, if $x\ge 1$, we have
\begin{equation}\label{bpdeinfty}
|F_z^{(j)}(x)|\lesssim_{j,\Re z} e^{2\pi |\Im z|} x^{-\Re z/2-1} e^{-x},\ x\ge 1 .
\end{equation}
\end{lemma}


\begin{proof} 
The estimates with $j=0$ follow directly from \eqref{upperboundofc} and \eqref{upperboundofk}. To prove the result with $j>0$, we begin with an observation of the relation between $F_z$ and its derivatives.
For $x>0$, the basic properties of the  Fourier transform imply that,
as $z\notin \R$,
\beeq\label{derivation}
\frac{d}{dx}F_z(x)=
\f (-i\eta\<\eta\>^{z})(x)
=
\f (\frac{\pa_\eta\<\eta\>^{z+2}}{(z+2)i})(x)
=-\frac{\ga(-z/2-1)}{2\ga(-z/2)} x F_{z+2}(x)
\ .
\eneq
It is then easy to see by the chain rule and induction that
\begin{equation}\label{bpderivative}
F_z^{(j)}(x)=\sum_{l=0}^{[ {j}/{2}]}C_l\frac{\ga(-z/2-(j-l))}{\ga(-z/2)}x^{j-2l}F_{z+2(j-l)}(x),
\end{equation}
where $C_l$ are constants from the chain rule.

By \eqref{eq-Notation-bessel} and \eqref{gaaa}, we find that
$$\left|\frac{\ga(-z/2-(j-l))}{\ga(-z/2)}C_{z+2(j-l)}\right|\les_{\Re z}
e^{\pi |\Im z|/2} 
$$
as $z\notin\R$. Then \eqref{bpdeinfty} follows from \eqref{bpderivative},
\eqref{eq-Notation-bessel}
 and \eqref{upperboundofk}.

Turning to the case $x\leq 1$, we will
also use
 \eqref{bpderivative} and \eqref{bpdezero} with $j=0$.
When $j$ is odd, thanks to $j-2l\geq1$,
each term in \eqref{bpderivative} is bounded 
by either
$x^{j-2l-1}\les 1$ ($\Re z+2(j-l)\le -1$) or
$x^{-\Re z-j-1}$ ($\Re z+2(j-l)> -1$), which, in turn, is bounded by
$1\vee x^{-\Re z-j-1}$.
Similarly, if $j$ is even, 
all of the terms in \eqref{bpderivative} are bounded 
by either
$x^{j-2l}\les 1$ ($\Re z+2(j-l)< -1$),
$x^{j-2l}(\frac{1}{|\Im z|}\wedge\ln\frac{1}{x}+1)$ ($\Re z+2(j-l)= -1$)
 or
$x^{-\Re z-j-1}$ ($\Re z+2(j-l)> -1$), which are controlled by
$1\vee x^{-\Re z-j-1}$, except the case when
$j=2l$ and $\Re z+2(j-l)= -1$.
For this exceptional case, $\Re z+j=-1$ and $j\in 2\N$, we have the estimate
$\frac{1}{|\Im z|}\wedge\ln\frac{1}{x}+1$ instead.

In summary, we get  \eqref{bpdezero}, which completes the proof.
\end{proof}

We conclude this section with the series representation of Bessel potentials and  the convergence of related series. For $0<x<\infty$, $\nu=\frac{-1-z}{2}$, by \cite[(7.5.2)-(7.5.3)]{MR2683157}, we have: 
\beq\label{series}
\begin{split}
F_z (x)&=C_z x^{\nu}K_{\nu}(x)=C_z \frac{\pi}{2\sin {\nu \pi}}x^{\nu}[I_{-\nu}(x)-I_{\nu}(x)]\\
&=C_z \frac{\pi}{2\sin {\nu \pi}}x^{\nu} \sum_{\pm, j\ge 0}^{\infty}\mp\frac{1}{j!\ga(\pm \nu+j+1)} \left(\frac x 2\right)^{\pm\nu+2j}\\
&=C_z \frac{\pi}{2\sin {\nu \pi}} [\sum_{j=0}^{\infty}d_j^1 x^{2j}-x^{2\nu}\sum_{j=0}^{\infty} d_j^2 x^{2j}]\\
&\triangleq C_z \frac{\pi}{2\sin {\nu \pi}} [g_z (x^2) +x^{2\nu} h_z(x^2)],
\end{split}
\eeq
where $d_j^1$, $d_j^2$ are defined as follows and stay the same in the sequel
$$
d_j^1=\frac{2^{\nu-2j}}{j!\ga(j+1-\nu)},
\ d_j^2=
\frac{2^{-\nu-2j}}{j!\ga(j+1+\nu)}
\ .$$

\begin{lemma}\label{thm-series}
Let $\nu=\mu +is\in\CC\backslash\R$, $0<r\leq 1$,
$$
g(r)=\sum_{j=1}^{\infty} d_j^1 C_j \sum_{l=0}^{j-1}C_{2j}^{2l+1} r^{l},\ 
h(r)=\sum_{j=0}^{\infty} d_j^2 r^{j},
$$
where $C_{2j}^{2l+1}=\frac{(2j)!}{(2l+1)!(2j-2l-1)!}$, and $|C_j|\leq c^{j}$ ($\forall j$)
for some fixed constant $c>0$. Then for any $p\in\SN$, the following estimates hold:
$$
|\partial_r^{(p)}g(r)|\leq C_{\mu,p} e^{\pi|s|},\  |\partial_r^{(p)}h(r)|\leq C_{\mu,p} e^{\pi|s|}\ .
$$
\end{lemma}

\begin{proof}
At
first, a simple calculation leads to 
\begin{equation*}
\sum_{l=0}^{j-1}  C_{2j}^{2l+1}=C_{2j}^1+C_{2j}^3+\cdots+C_{2j}^{2j-1}=\frac{(1+1)^{2j}-(1-1)^{2j}}{2}=2^{2j-1}.
\end{equation*}
Applying the fact $\ga(x+1)=x\ga(x)$ to $\ga(\pm\nu+j+1)$ for $j>j_0(\mu)$, ($j_0(\mu)$ is a constant such that $-\mu+j_0(\mu)=1$), we have the following estimate, uniformly with respect to $j$, 
\begin{equation*}
\left|\frac{1}{\ga(\pm\nu+j+1)}\right|\lesssim_{\mu} e^{\pi |s|}\ \ (\forall j\geq1),
\end{equation*}
in view of \eqref{gaaa}.

Given these preparations, we have
$$\partial_r^{(p)} g(r)=\sum_{j=1+p}^{\infty} d_j^1 C_j \sum_{l=p}^{j-1}C_{2j}^{2l+1} \frac{l!}{(l-p)!}r^{l-p},$$
and so
\begin{align*}
|\partial_r^{(p)} g(r)|&\leq \sum_{j=1+p}^{\infty} \left|
\frac{2^{\nu-2j}}{j!\ga(j+1-\nu)}
\right|c^j \sum_{l=p}^{j-1}C_{2j}^{2l+1} (l)(l-1)\cdots (l-p+1)\\
&\leq \sum_{j=1+p}^{\infty} \frac{1}{(j-p)!}\left|\frac{1}{\ga(j+1-\nu)} \right|2^{\mu}c^j\\
&\lesssim_{\mu,p} e^{\pi|s|}.
\end{align*}
A similar and even simpler argument
applies also to $h$ and this concludes the proof.
\end{proof}

\section{Wave kernel}\label{sec:wave}

Recall from Taylor \cite[section 8.5]{MR2743652} that,
if
$r=d_g(x,y)$ is the hyperbolic distance,
 $m(\tau)$ is an even function of $\tau$,
 and $\hat{m}$ is the Fourier transform of $m$,
then the kernel of the operator $m(D_0)$ is given by 
\begin{equation*}
(2\pi)^{-1} \left(-\frac{1}{2\pi \sinh r}\partial_r\right)^k \hat{m}(r)\ ,
\end{equation*}
when $n=2k+1$, and
\begin{equation*}
2^{-\hf}\pi^{-1}\int_r^{\infty}\frac{\sinh u}{(\cosh u-\cosh r)^{\hf}}\left(-\frac{1}{2\pi \sinh u}\partial_u\right)^k \hat{m}(u)du\ ,
\end{equation*}
when $n=2k$.
Thus, the kernels of the operators $S_z(t)$ and $C_z(t)$ with $z=-\rho+is$, given in \eqref{eq-Stein},
are multiples of the following
$$
(z+\rho)e^{z^2}\times\begin{cases}
(\shr \partial_r)^{k-1} (\frac{-i}{\sinh r} \int (\be^2 + \e^2)^{\hf z} \frac{\sin{t\e}}{\e}\e e^{-ir\e} d\e ),\\
(\shr \partial_r)^{k} (\int (\be^2 + \e^2)^{\hf (z-1)} \cos{t\e} e^{-ir\e} d\e ),
\end{cases}
$$
when $n=2k+1$, and
$$
(z+\rho)e^{z^2}\int_r^{\infty}\frac{\sinh u}{(\cosh u-\cosh r)^{\hf}}\times \begin{cases}
(\shu \partial_u)^{k-1} (\frac{-i}{\sinh u} \int (\be^2 + \e^2)^{\hf z} \frac{\sin{t\e}}{\e}\e e^{-iu\e} d\e )du,\\
(\shu \partial_u)^{k} (\int (\be^2 + \e^2)^{\hf (z-1)} \cos{t\e} e^{-iu\e} d\e )du.
\end{cases}
$$
when $n=2k$.

With the change of variable $\eta\to \beta\eta$, 
the proof of 
Theorem \ref{thm-main} is reduced to
the following
\begin{prop}\label{thm-reduce}
Let 
$n\ge 2$, $\rho=(n-1)/2$, $k=[n/2]\ge 1$, $t>0$,
$z=-\rho+is$ with  $s\in\R$,
 $\be\geq \rho$ for $n=3$ and $\be>\rho$ for $n \neq 3$.
Then
the following four functions 
\beeq
\label{prob} \Big|se^{-s^2}  (\shr \partial_r)^{k-1} (\shr \int_{-\infty}^{\infty} \<\e\>^{-k+is} \sin \be t\e \ e^{-i\be r \e}d\e)\Big|,\eneq
\beeq\label{probb}\Big|se^{-s^2}  (\shr \partial_r)^{k} \int_{-\infty}^{\infty}
\<\e\>^{-k-1+is}  \cos{\be t\e} \ e^{-i\be r\e} d\e \Big|,\eneq
\beeq
\label{probbb} \Big| \int_r^{\infty}\frac{se^{-s^2}\sinh u}{(\cosh u-\cosh r)^{\hf}}(\shu \partial_u)^{k-1} \frac{ \int 
\<\e\>^{1/2-k+is}  \sin \be t\e \ e^{-i\be u \e} d\e}{\sinh u} du\Big|\ ,\eneq
\beeq\label{probbbb}\Big|  \int_r^{\infty}\frac{se^{-s^2}\sinh u}{(\cosh u-\cosh r)^{\hf}}(\shu \partial_u)^{k} 
\int\<\e\>^{-k-1/2+is} 
\cos{\be t\e}\ e^{-i\be u\e} d\e )du\Big|\ ,\eneq
are 
 uniformly (with respect to $r\geq0, s\in\R$) bounded by
 $(\sinh t)^{-\rho}$.
\end{prop}

Before presenting the proof of
Proposition \ref{thm-reduce}, we 
do some preparations concerning
the operator
$\shr\partial_r$, as well as 
its action on functions. 
Let
$\tau=\cosh r-1$ be a smooth change of variable, we see that
$\partial_{\tau}=\shr\partial_r$.
As $\frac{\sinh r}{r}$ is a smooth function of $r^2$ and $r^2$ is equivalent to $\tau$ near the origin ($r^2$ is a smooth function of $\tau$ for $\tau\ge 0$), we see that
if $f(\tau)\triangleq\frac{r}{\sinh r}$, then
\begin{equation}\label{fderiv}
\partial_{\tau}^p f(\tau)=(\shr\partial_r)^p\frac{r}{\sinh r}=\mathcal{O}_p(1)\ , \ r\les 1,\ p\in\SN\ .
\end{equation}
In addition, a simple computation deduces that 
\begin{equation}\label{sinhr}
\left|\partial_{\tau}^{p}(\shr)\right|=\left|\frac{\sum_{l=0}^{[\hf p]}C_{p,l} \cosh^{p-2l}r\sinh^{2l}r}{(\sinh r)^{2p+1}}\right|\lesssim_p \begin{cases}
r^{-(2p+1)},& r\les 1\ ,\\
(\sinh r)^{-p-1}, & r\gtrsim 1\ .
\end{cases}
\end{equation}
In fact, we mainly use \eqref{fderiv} and \eqref{sinhr} to absorb or even eliminate the influence of ``$r^{-1}$" in the following situations. 

When it comes to the general function $g(r)\in C^{\infty}$, we have
\beq\label{onlyg}
\begin{split}
\left|\partial_{\tau}^p g\right|&\lesssim_p \sum_{j=1}^{p} \sum_{\sum \al_k=p-j} |\partial_r^jg(r)|\times\prod_{k=1}^p\left|\partial_r^{\al_k} (\sinh r)^{-1}\right|  \\
&\lesssim_p\begin{cases}
\sum_{j=1}^{p}r^{-(2p-j)}|\partial_r^jg(r)|& r\les 1\ ,\\
(\sinh r)^{-p}\sum_{j=1}^p|\partial_r^jg|&r\gtrsim 1\ .
\end{cases}
\end{split}
\eeq 
A direct application of \eqref{onlyg} shows 
\beq\label{generalfunc}
\begin{split}
\left|\partial_{\tau}^p(\shr g)\right|&=\left|\sum_{l=0}^pC_l\partial_{\tau}^{p-l}(\shr)\partial_{\tau}^l(g)\right|\\
&\lesssim_p\begin{cases}
\sum_{j=0}^p r^{-(2p-j+1)}|\partial_r^jg|& r\les 1\ ,\\
(\sinh r)^{-p-1}\sum_{j=0}^p|\partial_r^jg|&r\gtrsim 1\ ,
\end{cases}
\end{split}
\eeq
where $C_l$ are constants from chain rule and by induction.
When ``$\shr$" is replaced by the well-behaved ``$f(\tau)$" for  {$r\les 1$}, a better estimate is available
\beq\label{fgderiv}
\begin{split}
\left|\partial_{\tau}^p (f(\tau)g(r))\right|&=\left|\sum_{l=0}^p C_l \partial_{\tau}^{p-l}f\partial_{\tau}^lg(r)\right|\\
&\lesssim_{p}\sum_{l=1}^p \left|(\shr\partial_{r})^l g(r)\right|+|g(r)|\\
&\lesssim_p \sum_{j=1}^{p} r^{-(2p-j)} |\partial_r^j g(r)|+|g(r)|.
\end{split}
\eeq

In particular,  when $g_l$ ($l\in\mathbb{N}$) has the following special structure, for $r\les 1$,
\begin{equation}\label{specialg}
\partial_rg_l(r)=rg_{l+1}(r)=\sinh r f(\tau)g_{l+1}(r),
\end{equation}
that is, $\partial_{\tau}g_l(r)=f(\tau)g_{l+1}(r)$,
then we have 
\beq\label{specialgg}
\begin{split}
|\partial_{\tau}^l(f(\tau)g_0(r))|&\lesssim_l \sum_{j=1}^{l}|\partial_{\tau}^j g_0(r)|+|g_0(r)|\\
&\lesssim_l\sum_{j=0}^{l-1}|\partial_{\tau}^j(f(\tau)g_1(r))|+|g_0(r)|\\
&\lesssim_l \sum_{j=1}^{l-1}|\partial_{\tau}^jg_1(r)|+|g_0(r)|+|g_1(r)|\\
&\lesssim_l \sum_{j=0}^l|g_j(r)|.
\end{split}
\eeq
For example, $g(r^2)$ and \eqref{integralderiv} have such structure in \eqref{specialg}.


\section{Proof of Proposition \ref{thm-reduce}}\label{sec:prf}
In this section, we give the proof of Proposition \ref{thm-reduce}.

\subsection{Consider \eqref{prob}}\label{oddsin}
\subsubsection{\bm{$r\geq1$}}\label{oddsinbigr}

By \eqref{generalfunc} with the case $r\ge 1$, we have the following estimate
\begin{align*}
\eqref{prob}&\lesssim_{\be} \sum_{m=0}^{k-1}(\sinh r )^{-k}\left|se^{-s^2}  \partial_{r}^m\int \<\e\>^{-k+is} \sin \be t\e e^{-i\be r \e}d\e\right|\\
&\les \sum_{m=0}^{k-1}\sum_{\pm} |s e^{-kr-s^2} F_{-k+is}^{(m)}(\be (r\pm t))|
\ .
\end{align*}

Concerning $F_{-k+is}^{(m)}$, we have the 
following consequence of Lemma \ref{lem-2.1}
\begin{lemma}\label{lem-2.3}
Let $0\leq m\leq k-1$, $\be\geq k$ for $k=1,2$ and $\be>k$ for $k\geq3$,
Then there exists a constant $C$,  independent of  $s\in\R$, $\la$, such that
\begin{equation*}
|s e^{-s^2}F_{-k+is}^{(m)}(\be \la)|
\leq Ce^{-k|\la|}. 
\end{equation*}
\end{lemma}

\begin{proof}
For the case of $|\la|\ge 1$,
\eqref{bpdeinfty} in  Lemma \ref{lem-2.1} 
gives us 
$$|s e^{-s^2} F_{-k+is}^{(m)}(\be\la)|\les_m
|s| e^{2\pi |s|-s^2} |\be\la|^{k/2-1} e^{-|\be\la|}$$
 which could be controlled by $e^{-k|\la|}$, thanks to our assumption on the relation between $\be$ and $k$.

On the other hand, if $0<|\la|\le 1$, as we are assuming $m\le k-1$, 
\eqref{bpdezero} in  Lemma \ref{lem-2.1} tells us that
$$|s  e^{-s^2} F_{-k+is}^{(m)}(\be\la)|\les_m
|s| e^{2\pi |s|-s^2} 
\left(1+\frac{1}{|s|}\right)\les 1\ .
$$

As $\f^{-1}(F_{-k+is}^{(m)})(\eta)=(2\pi)^{-1}(-i\eta)^m\<\eta\>^{-k+is}$, which is integrable for $m< k-1$, the only remaining case is 
$m=k-1$ near $\la=0$.
To treat this situation, we recall
 the corresponding Riesz potential
$$R_{-k+is}=\f(|\e|^{-k+is})(\la)
=2^{-k+is}\pi^{1/2}\frac{\ga((-k+1+is)/2)}{\ga((k-is)/2)}|\la|^{k-1-is}
$$
where $|\e|^{-k+is}\in C^\infty(\R\backslash \{0\})\cap \mathcal{S}'(\R)$ is a homogeneous distribution of degree $-k+is$, see, e.g., \cite[\S 3.8 (8.32-8.36)]{MR1395148}.

The relation between the Riesz potential and the Bessel potential naturally yields the desired result, in view of the uncertainty principle.
To be specific, let us give the proof.
It is clear that
$R_{-k+is}$ is  in $C^{k-1}$ for $s\neq 0$.
Let
$\phi\in C_c^\infty(\R)$ which is identity near $0$, and $\psi=1-\phi$,
then we have
\begin{eqnarray*}
\<\e\>^{-k+is} & = & |\e|^{-k+is}
-\phi
|\e|^{-k+is}
+\phi  \<\e\>^{-k+is}+\psi (\<\e\>^{-k+is}-|\e|^{-k+is}) \\
 & = &  |\e|^{-k+is}
-\phi
|\e|^{-k+is}+\<\eta\>^{1-k}L^1\ ,
\end{eqnarray*}
and so, modulo a $C^{k-1}$ function,
$$F_{-k+is}=\f(\phi  |\e|^{-k+is})=c \f(\phi)*R_{-k+is}\in C^{k-1}\ .$$
This completes the proof.
\end{proof}
Equipped with  Lemma \ref{lem-2.3}, we could proceed with the following 
\begin{align*}
\eqref{prob}
&\les \sum_{m=0}^{k-1}\sum_{\pm} |s e^{-kr-s^2} F_{-k+is}^{(m)}(\be (r\pm t))|
\\
&\lesssim_{k} e^{-kr} (e^{-k|r-t|}+e^{-k|r+t|}) \les e^{-kt}\lesssim (\sinh t)^{-k}\ ,
\end{align*}
which give us the desired result.

\subsubsection{\bm{$0\leq r<1,r\leq\hf t$}}\label{oddsecond}

To begin the proof, we observe that  $\eqref{prob}$ could be written as
\begin{equation*}
\begin{split}
\eqref{prob}&= \left|\hf se^{-s^2}(\shr \partial_r)^{k-1} (\shr [F_z(\be(t-r))-F_z(\be(t+r))])\right|\\
&=\left|\hf\be se^{-s^2}(\shr \partial_r)^{k-1} (\frac{r}{\sinh r} \int_{-1}^1 F_z^{\prime}(\be(t+\theta r)) d\theta)\right|\ ,
\end{split}
\end{equation*}
where $z=-k+is$.
Observing that
\begin{eqnarray}
\partial_r \int_{-1}^1 F_z^{\prime}(\be(t+\theta r)) d\theta & = & 
\be \int_{-1}^1 (F_z^{\prime \prime} (\be(t+\theta r))-F_z^{\prime \prime}(\be t))\theta d\theta  \label{integralderiv}\\
 & = & r \be^2 \int_{-1}^1 \int_0^{\theta} F_z^{(3)}(\be(t+\theta_1 r))d\theta_1 \theta d\theta, \nonumber
\end{eqnarray}
which
has the same structure as $\eqref{specialg}$.
With the help of \eqref{integralderiv},
  by \eqref{specialgg}, we can then control  $\eqref{prob}$ as follows
\begin{equation*}
\eqref{prob}\lesssim_{\be,k}\sum_{l=0}^{k-1}\left|se^{-s^2}\int_{-1}^1\int_0^{\theta_0}\cdots\int_0^{\theta_{l-1}} F_z^{(2l+1)}(\be(t+\theta_l r)) d\theta_l \theta_{l-1} d\theta_{l-1}\cdots \theta_0 d\theta_0\right|.
\end{equation*}

Then, it is easy to employ Lemma \ref{lem-2.1} to obtain when $t\geq 2$:
\begin{align*}
\eqref{prob}\lesssim_{\be,k} (t-r)^{k/2-1} e^{-\be (t-r)}\lesssim_{\be,k}e^{-kt}\sim_k (\sinh t)^{-k},
\end{align*}
and when $0<t<2$:
\begin{align*}
\eqref{prob}\lesssim_{\be,k} (t-r)^{k-(2k-1)-1}\lesssim_k t^{-k}\sim_k (\sinh t)^{-k}.
\end{align*}
In particular, we observe that when $k=1,2$, $\be$ can be  $k$, while for $k>2$, $\be>k$ is required.

\subsubsection{\bm{$0\leq r<1,r\geq\hf t$}}\label{oddsmall}

In this case, the asymptotic behavior of the Bessel potential does not seem to be helpful, due to the troublesome factor ``$t-r$" when $k\ge 2$. 
Instead, we try to use an alternative formula of $F_z(\be|t-r|)-F_z(\be(t+r))$ from the perspective of series expansion,

By \eqref{series}, we can split ``$F_z$" into two pieces involving  $g_z$ and  $h_z$, as follows
\begin{align*}
\eqref{prob}=&\left|\hf se^{-s^2}(\shr \partial_r)^{k-1} (\shr [F_z(\be|t-r|)-F_z(\be(t+r))])\right|\\
=&\Big|  C_z \frac{\pi}{4\sin {\nu \pi}} se^{-s^2}(\shr \partial_r)^{k-1}\Big\{\shr \Big[(g_z ((\be(t-r))^2)-g_z((\be(t+r))^2))\\
&+((\be|t-r|)^{2\nu}h_z((\be(t-r))^2)-(\be(t+r))^{2\nu}h_z((\be(t+r))^2))\Big]\Big\}\Big|\\
\triangleq&|\one+\two|\ ,
\end{align*}
where $\nu=\frac{-1-z}{2}=\frac{k-1-is}2$. 
Notice that,
when $k=1$,  
$sF_z$ are bounded near $0$,
by
 Lemma \ref{lem-2.3},
 which gives us
$$
 \eqref{prob}\les \shr\les r^{-1}\les t^{-1}\simeq (\sinh t)^{-k}
 \ , 
 k=1\ .$$
In the following, we assume $k\ge 2$.

Clearly, the part involving $g_z$ can form in terms of \eqref{specialg}
\begin{align*}
|\one|&=|C_z \frac{\pi s}{4\sin {\nu \pi}} e^{-s^2}(\shr \partial_r)^{k-1}(\shr [g_z ((\be(t-r))^2)-g_z((\be(t+r))^2)])|\\
&=|C_z \frac{\pi s}{2\sin {\nu \pi}}  e^{-s^2}(\shr \partial_r)^{k-1}(\frac{r}{\sinh r}\sum_{j=1}^{\infty} d_j^1 \be^{2j} \sum_{l=0}^{j-1}C_{2j}^{2l+1} r^{2l}t^{2j-2l-1})|\\
&\triangleq |C_z \frac{\pi s}{2\sin {\nu \pi}}  e^{-s^2}(\shr \partial_r)^{k-1}(\frac{r}{\sinh r}g(r^2))|.
\end{align*}
With the help of \eqref{specialgg}
and \eqref{upperboundofc},
 the proof of Lemma \ref{thm-series} with $\nu=\hf(k-1-is)$
gives us
$$|\one|\les_k
\left|\frac{ s}{ \sin {\nu \pi}} \right|  e^{2\pi|s|-s^2}\les 1\lesssim t^{-k}\ ,$$
where the factor $s$ is used to absorb the possible 
bad term appeared in $|\sin \nu\pi|$ when $k$ is odd and $s$ is small.

Heuristically, the term $\two$ behaves better and is expected easier to control. However,
it turns out that this term is a little more difficult to handle, so that we could avoid the possible occurrence of the bad term like $|t-r|^{-1}$.

\bm{$\two.1$} ($\hf t\leq r\leq t$). Using \eqref{fgderiv}, we deduce that
\begin{eqnarray*}
|\two|&=&\left| C_z \frac{\pi s}{4\sin {\nu \pi}} e^{-s^2}(\shr \partial_r)^{k-1}(\shr (\be x)^{2\nu}h_z((\be x)^2)|_{x=t+r}^{x=t-r})\right|\\
&=&\left| C_z \frac{\pi s}{4\sin {\nu \pi}} e^{-s^2}(\shr \partial_r)^{k-1}(\shr  \be^{2\nu}\int_{-1}^1 r\partial_x(x^{2\nu}h_z((\be x)^2))|_{x=t+\theta r}  d\theta)\right|\\
&\triangleq &\left|  C_z \be^{2\nu} \frac{\pi s}{4\sin {\nu \pi}} e^{-s^2}(\shr \partial_r)^{k-1}(\frac{r}{\sinh r} \hz(r,t))\right|\\
&\lesssim_{\be,k}&\sum_{l=0}^{k-1} \left|C_z e^{-s^2}r^{-(2(k-1)-l)}\partial_r^{l}\hz(r,t)\right|.
\end{eqnarray*}
By uniformly convergence for $h_z$ in Lemma  \ref{thm-series}, we find that the upper bound for $\partial_r^l \hz(r,t)$ is essentially determined by the main item $$\int_{-1}^1 (\partial_r^{l} (t+\theta r)^{2\nu-1}) h_z((\be(t+\theta r))^2) d\theta$$ for small $r,t$, 
while the remaining  items are easier to control, due to the occurrence of the higher order of ``$t+\theta r$".

When $l<k-1$, we have $\Re (2\nu -1-l)=k-l-2\geq 0$, then
$$|\partial_r^l\hz|\lesssim_k \<s\>^l e^{\pi |s|} (t+r)^{k-l-2}\ .$$
When $l=k-1\ge 1$, 
we observe that
$$\partial_r^{k-1} (t+\theta r)^{k-2-is}=P(k,s)(-is)\theta^{k-1}
(t+\theta r)^{-1-is}=P(k,s)\frac{\theta^{k-1}}{ r}\pa_\theta (t+\theta r)^{-is}\ ,$$
where $|P(k,s)|\les \<(k-2,s)\>^{k-2}$.
Then, integration by parts with respect to $\theta$, together with
Lemma  \ref{thm-series},
 gives us that
$$
|\partial_r^{k-1}\hz|\lesssim_k \<s\>^{k-2} e^{\pi |s|} r^{-1}\ .
$$

In summary, together with
 \eqref{upperboundofc},
 these estimates imply the desired bound
\begin{equation*}
|\two| \lesssim_{\be,k} \sum_{l=0}^{k-2}r^{-(2(k-1)-l)}(t+r)^{k-l-2}+r^{-(k-1)}r^{-1}\lesssim  t^{-k}.
\end{equation*}

\bm{$\two.2$} ($r\geq t$). Using \eqref{generalfunc}, we have
\begin{align*}
|\two|&=\left|  C_z \frac{\pi s}{4\sin {\nu \pi}} e^{-s^2}(\shr \partial_r)^{k-1}(\shr \be^{2\nu}\int_{-1}^1 t\partial_x(x^{2\nu}h_z((\be x)^2))|_{x=r+\theta t}  d\theta)\right|\\
&\triangleq\left| C_z \be^{2\nu} t \frac{\pi s}{4\sin {\nu \pi}} e^{-s^2}(\shr \partial_r)^{k-1}(\frac{1}{\sinh r} \hzz(r,t))\right|\\
&\lesssim_{\be,k}\sum_{l=0}^{k-1}\left|C_z e^{-s^2}t r^{-(2k-l-1)} \partial_r^l\hzz(r,t)\right|.
\end{align*}
Using
 the fact $t\partial_r=\partial_{\theta}$ for $f(r+\theta t)$,
a similar argument as in $\two.1$  yields
\begin{equation*}
|\two|\lesssim_{\be,k}\sum_{l=0}^{k-2}tr^{-(2k-l-1)} (r+t)^{k-l-2}+r^{-k} \lesssim_k t^{-k}\ ,
\end{equation*}
which completes the proof for the case $1>r>t/2>0$.

\subsection{Consider $\eqref{probb}$}\label{oddcos} 
Comparing with Subsection \ref{oddsin}, the main difference is the appearance of cosine function instead of the sine function, which mainly affects the part of proof involving series expansion.

\subsubsection{\bm{$r\geq1$}} In this case, $z-1=-k-1+is,\,1\leq m\leq k$, it follows from the same argument as that in  Subsection \ref{oddsinbigr}.

\subsubsection{\bm{$0\leq r<1,r\leq \hf t$}} Similar to Subsection \ref{oddsecond}, we have
\begin{align*}
&\partial_r\int_{-1}^1F_{z-1}^{\prime\prime}(\be(t+\theta r))d\theta\\
=&r\be^2\int_{-1}^1\int_0^{\theta}F_{z-1}^{(4)}(\be(t+\theta_1 r))d\theta_1\theta d\theta,
\end{align*}
satisfies \eqref{specialg}. Then by \eqref{specialgg} and $F_z(x)=F_z(-x)$ one has
\begin{equation*}
\begin{split}
&\eqref{probb}=\left|\hf\be se^{-s^2}(\shr\partial_{r})^{k-1}(\shr (F_{z-1}^{\prime}(\be(t+r))-F_{z-1}^{\prime}(\be(t-r))))\right|\\
&=\left|\hf\be^2 se^{-s^2}(\shr\partial_{r})^{k-1}(\frac{r}{\sinh r} \int_{-1}^1 F_{z-1}^{\prime\prime}(\be(t+\theta_0 r))d\theta_0)\right|\\
&\lesssim_{\be,k}\sum_{l=0}^{k-1}\left|se^{-s^2}\int_{-1}^1\int_0^{\theta_0}\cdots\int_0^{\theta_{l-1}}F_{z-1}^{(2l+2)}(\be(t+\theta_l r))d\theta_l\theta_{l-1}d\theta_{l-1}\cdots\theta_0d\theta_0\right|.
\end{split}
\end{equation*}
Then,
by Lemma \ref{lem-2.1}, we conclude this case when $t\geq 2$:
$$
\eqref{probb}\lesssim_{\be,k} (t-r)^{\frac{k-1}{2}}e^{-\be (t-r)}\lesssim e^{-kt}\sim(\sinh t)^{-k}\ ,
$$
and when $0<t<2$:
\begin{equation*}
\eqref{probb}\lesssim_{\be,k}(t-r)^{k-2(k-1)-2} \lesssim  t^{-k}\sim (\sinh t)^{-k}\ .
\end{equation*}
\begin{remark}
For the part with $0\leq r<1$, $t\geq 2$, we find that $\be\geq k$ is admissible only when $k=1$, 
 which is different from \eqref{prob}. 
\end{remark}

\subsubsection{\bm{$0\leq r<1,r\geq \hf t$}}

As in Subsection \ref{oddsmall}, 
by \eqref{series}, 
we consider \eqref{probb} from the perspective of series expansion,
\begin{align*}
\eqref{probb}=&\Big|\hf se^{-s^2}(\shr\partial_r)^k\left[F_{z-1}(\be|t-r|)+F_{z-1}(\be(t+r))\right]\Big|\\
=&\Big|C_z \frac{\pi s}{4\sin {\nu \pi}} e^{-s^2}(\shr \partial_r)^{k} \Big[(g_{z-1} ((\be(t-r))^2)+g_{z-1}((\be(t+r))^2))\\
+&((\be|t-r|)^{2\nu}h_{z-1}((\be(t-r))^2)+(\be(t+r))^{2\nu}h_{z-1}((\be(t+r))^2))\Big]\Big|\\
\triangleq&|\three+\four|,
\end{align*}
where $\nu=\hf(-z)=\hf(k- is)$.

The
estimate of $\three$ follows the same way as that of $\one$
\begin{align*}
|\three|&=\Big|C_z \frac{\pi s}{2\sin {\nu \pi}} e^{-s^2}(\shr \partial_r)^{k}\sum_{j=0}^{\infty}d_j^1 \be^{2j}\sum_{l=0}^{j}C_{2j}^{2l}r^{2l}t^{2j-2l}\Big|\\
&\triangleq\Big| C_z \frac{\pi s}{2\sin {\nu \pi}} e^{-s^2} \partial_{\tau}^{k}g^1(r^2)\Big|\lesssim_k 1\lesssim_k t^{-k}.
\end{align*}

Concerning \four,
similar as that of \two, 
the case $k=1$ is trivial in view of
 Lemma \ref{lem-2.3} and we consider the remaining case $k\ge 2$.

\bm{$\four.1$} ($\hf t\leq r\leq t$). For this case, we operate $\partial_r$ once to get
\begin{equation*}
\begin{split}
|\four|=&\Big|C_z \frac{\pi s}{4\sin {\nu \pi}} e^{-s^2}\be^{2\nu}(\shr\partial_r)^{k-1}\Big(\shr\cdot \\
&\Big[2\nu((t+r)^{2\nu-1}h_{z-1}((\be(t+r))^2)-(t-r)^{2\nu-1}h_{z-1}((\be(t-r))^2))\\
&+2\be((t+r)^{2\nu+1}h_{z-1}^{\prime}((\be(t+r))^2)-(t-r)^{2\nu+1}h_{z-1}^{\prime}((\be(t-r))^2))\Big]\Big)\Big|.
\end{split}
\end{equation*}
Then
we use the fundamental theorem of calculus to extract the favorite term $\frac{r}{\sinh r}$, and thus consider the main item, with
$$\left|C_z \frac{\pi s}{4\sin {\nu \pi}} e^{-s^2}
\int_{-1}^1 r^{-k}(\partial_r^{k-1}(t+\theta r)^{k-2-is})h_{z-1}((\be(t+\theta r))^2)d\theta\right| \lesssim_k r^{-k}\sim t^{-k}\ .$$

\bm{$\four.2$} ($r\geq t$). In this case, ``$\partial_r$" cannot bring the favorite form to use the fundamental theorem of calculus. To remedy this issue, we introduce an artificial error term to extract the desired form.
\begin{align*}
|\four|\leq&\Big|C_z \frac{\pi s}{4\sin {\nu \pi}}\be^{2\nu} e^{-s^2}\cdot(\shr \partial_r)^{k}\\
&\left[(r-t)^{2\nu}h_{z-1}((\be(r-t))^2)-(r+t)^{2\nu}h_{z-1}((\be(r+t))^2)\right]\Big|\\
&+\Big|2C_z \frac{\pi s}{4\sin {\nu \pi}}\be^{2\nu} e^{-s^2}(\shr \partial_r)^{k}\left[(r+t)^{2\nu}h_{z-1}((\be(r+t))^2)\right]\Big|. 
\end{align*}

The first item of the right hand side could be handled similarly as \four.1. Concerning the second item,  using
 \eqref{onlyg} and Lemma  \ref{thm-series}, we see that
it is an admissible error term:
\begin{align*}
&\Big|(\shr\partial_r)^k \left[(r+t)^{2\nu}h_{z-1}((\be(r+t))^2)\right]\Big|\\
\lesssim_{\be,k}&\sum_{1\leq j\leq k,l\leq j} r^{-(2k-j)}\Big|(\partial_r^{j-l}(r+t)^{k-is})\partial_r^lh_{z-1}((\be(r+t))^2)\Big|\\
\lesssim_{\be,k}&\<s\>^k r^{-k}e^{\pi |s|}\lesssim  e^{2\pi |s|} t^{-k}.
\end{align*}


\subsection{Consider \eqref{probbb}}\label{evensin} 

When $n=2k$ with $z=-\rho+is=1/2-k+is$,  the situation is more complicated to treat, due to the much more involved expression of  \eqref{probbb},
compared with \eqref{prob}.
In any case, it turns out that the similar argument as in Subsection \ref{oddsin} still apply.

\subsubsection{\bm{$r\geq1$}}\label{evenbig}

By \eqref{onlyg}, we need only to consider
\begin{equation*}
\begin{split}
\eqref{probbb}&\lesssim \sum_{l=0}^{k-1}\sum_{\pm}\sum_{\sum\al_j=k-1-l}\Big|se^{-s^2}\int_r^{\infty}(\prod_{1}^k \partial_u^{\al_j}\shu)\partial_u^lF_z(\be|u\pm t|)\frac{\sinh u}{(\cosh u-\cosh r)^{\hf}}du\Big|\\
&\triangleq \five_+ +\five_-.
\end{split}
\end{equation*}
For the term with plus sign,
since $u+t\geq 1$, with the help of  \eqref{bpdeinfty} in Lemma \ref{lem-2.1}, $\five_+$ could be easily bounded:
\begin{equation*}
\begin{split}
\five_+&\lesssim_{\be,n}\Big|se^{2\pi|s|-s^2}\int_r^{\infty}(\shu)^{k}(u+t)^{\hf(k-\frac{5}{2})}e^{-\be(u+t)}\frac{\sinh u}{(\cosh u-\cosh r)^{\hf}}du\Big|\\
&\lesssim_n\Big|se^{-\hf s^2}e^{-(k-\hf)t}\int_r^{\infty}e^{-(k-1)u}e^{-(k-\hf)u}\frac{1}{(\cosh u-\cosh r)^{\hf}}du\Big|\\
&\lesssim (\sinh t)^{-k+\hf}\ ,
\end{split}
\end{equation*}
as $\be>\rho$.

On the other hand, the term with minus sign is much more involved to control, for which we 
split the integral into several cases, depending on the size of
$u$ compared with $t$. Let us begin with the easier case with $|u-t|\geq1$, for which  it is clear that, for any $\rho<\be_0<\be$, $$|\partial_u^lF_z(\be|u-t|)|\lesssim_{\be,k}|u-t|^{\hf(k-\frac{5}{2})}e^{-\be|u-t|}\lesssim_k e^{-\be_0|u-t|}\leq 1\ .$$
Then, if $u\geq t+1$, we get
\begin{equation*}
\begin{split}
&\int_{r\vee (t+1)}^{\infty}(\shu)^{k}\frac{\sinh u}{(\cosh u-\cosh r)^{\hf}}du\\
\lesssim&\int_{r\vee (t+1)}^{(r+1)\vee (t+2)}(\shu)^{k}\frac{\sinh u}{(\cosh u-\cosh r)^{\hf}}du+\int_{(r+1)\vee (t+2)}^{\infty}e^{(-k+\hf)u}du\\
\lesssim &e^{-k (r\vee (t+1))}(\cosh u-\cosh r)^{\hf}|_{u=r\vee (t+1)}^{(r+1)\vee (t+2)}+e^{(-k+\hf)t}\\
\lesssim &(\sinh t)^{-k+\hf}\ ,
\end{split}
\end{equation*}
where we have used the fact that
$\cosh u-\cosh r\gtrsim \cosh u\gtrsim e^u$ if $u\ge r+1$, as well as 
$\pa_u(\cosh u-\cosh r)^{\hf}=\hf(\cosh u-\cosh r)^{-\hf}\sinh u$.
On the other hand, if $u\leq t-1$, we have 
$\cosh u-\cosh r\simeq e^r (u-r)$ for $u\in [r, r+1]$ with $r\ge 1$,
and so
\begin{equation*}
\begin{split}
&\int_{r}^{t-1}(\shu)^{k}\frac{\sinh u}{(\cosh u-\cosh r)^{\hf}}e^{-\be_0(t-u)}du \\
\lesssim &e^{-\be_0t} \int_r^{(t-1)\wedge (r+1)}e^{(\be_0-\rho)r}(u-r)^{-\hf}du + e^{-\be_0t}\int_{(t-1)\wedge (r+1)}^{t-1}e^{(\be_0-\rho)u}du \\
\lesssim &(\sinh t)^{-\rho}\ .
\end{split}
\end{equation*}

Turning to the case of $|u-t|\le 1$.
By  \eqref{bpdezero} in Lemma \ref{lem-2.1}, $|\partial_u^lF_z(\be |u-t|)|\lesssim_k e^{2\pi |s|}$ if $l<k-1$ and $u\neq t$.
As $\f^{-1}(F_{-k+\hf+is}^{(l)})(\e)=(2\pi)^{-1}(-i\e)^l\<\e\>^{-k+\hf+is}\in L^1(\R)$ for $l<k-1$, the estimate applies also for $u=t$.
Thus,
 it remains to deal with $l=k-1$:
\beq\label{mainpart}
\begin{split}
&\int_{\{u\geq r,|u-t|\leq1\}}(\shu)^{k}\frac{\sinh u}{(\cosh u-\cosh r)^{\hf}}du\\
+&\Big|se^{-s^2}\int_{\{\cdots\}}(\shu)^{k}\frac{\sinh u}{(\cosh u-\cosh r)^{\hf}} \partial_u^{k-1}\f(\<\e\>^{-k+\hf+is})(\be(u-t)) du\Big|,
\end{split}
\eeq
Here,
as in the proof of Lemma \ref{lem-2.3},
 the term $\partial_u^{k-1}\f(\<\e\>^{-k+\hf+is})(\be(u-t))$
could be reduced to $\partial_u^{k-1}\f(|\e|^{-k+\hf+is})(\be(u-t))$, which is a homogeneous distribution of degree $-\hf-is$.
That is,
\begin{equation*}
\begin{split}
\eqref{mainpart}\lesssim_{\be,n}&\int_{\{u\geq r, |u-t|\le 1\}}(\shu)^{k}\frac{\sinh u}{(\cosh u-\cosh r)^{\hf}}du\\
&+\Big|se^{-\hf s^2}\int_{\{\cdots\}}(\shu)^{k}\frac{\sinh u}{(\cosh u-\cosh r)^{\hf}} |u-t|^{-\hf-is}du\Big|\\
\triangleq&A+B.
\end{split}
\end{equation*}
The first term is easy to control, as $r<t+1$,
\begin{align*}
A&=\int_{r\vee t-1}^{t+1}(\shu)^{k}\frac{\sinh u}{(\cosh u-\cosh r)^{\hf}}du\\
&\lesssim\int_{r\vee t-1}^{t+1}e^{-ku}\frac{\sinh u}{(\cosh u-\cosh r)^{\hf}}du\\
&\lesssim e^{-kt}(\cosh (t+1))^{\hf}\\
&\lesssim_n (\sinh t)^{-k+\hf}.
\end{align*}
Finally, we treat the delicate term $B$.\\
\uline{\bm{$B.1(0\leq r-t\leq \hf)$}:}
\begin{align*}
B=&\Big|se^{-\hf s^2}\int_r^{t+1}(\sinh u)^{-k}(u-t)^{-\hf-is}d(\cosh u-\cosh r)^{\hf}\Big|\\
\lesssim_k&\Big|(\sinh u)^{-k} (u-t)^{-\hf-is}(\cosh u-\cosh r)^{\hf}|_{u=r}^{t+1}\Big|\\
&+\int_r^{t+1}(\sinh u)^{-k-1}\cosh u (\cosh u-\cosh r)^{\hf}(u-t)^{-\hf}du\\
&+\Big|se^{-\hf s^2}\<s\>\int_r^{t+1}(\sinh u)^{-k} (\frac{\cosh u-\cosh r}{u-t})^{\hf}(u-t)^{-1-is}du\Big|\\
\lesssim&(\sinh t)^{-k+\hf}+\Big|e^{- s^2/3}\int_r^{t+1}(\sinh u)^{-k} (\frac{\cosh u-\cosh r}{u-t})^{\hf}\partial_u(u-t)^{-is}du\Big|\\
\lesssim&(\sinh t)^{-k+\hf}+\int_r^{t+1}(\sinh u)^{-k}\partial_u(\frac{\cosh u-\cosh r}{u-t})^{\hf}du\\
\lesssim&(\sinh t)^{-k+\hf}+(\sinh t)^{-k}\int_r^{t+1}\partial_u(\frac{\cosh u-\cosh r}{u-t})^{\hf}du\lesssim(\sinh t)^{-k+\hf},
\end{align*}
where we have used the fact that $\partial_u(\frac{\cosh u-\cosh r}{u-t})^{\hf}$ is positive.
\\
\uline{\bm{$B.2(r-t\geq \hf)$}:} 
Although this case could be included in B.1,
we choose to present it with a simpler proof. Actually, by mean value theorem, we have
\begin{align*}
B\lesssim (\sinh t)^{-k+\hf}\int_r^{t+1}(u-r)^{-\hf}(u-t)^{-\hf}du\lesssim (\sinh t)^{-k+\hf}.
\end{align*}
\uline{\bm{$B.3(0\leq t-r\leq\hf)$}:} Similar as $B.1$, we find that $\partial_u(\frac{u-t}{\cosh u-\cosh r})^{\hf}$ is positive in $[t,t_0]$ and negative in $[t_0,t+1]$ with some $t_0\in [t,t+1]$. Equipped with this information, we get
\begin{align*}
B\leq&(\sinh r)^{-k+\hf}\int_r^{t}(u-r)^{-\hf}(t-u)^{-\hf}du\\
&+\Big|e^{-\hf s^2}\int_t^{t+1}
(\sinh u)^{1-k}
(\frac{u-t}{\cosh u-\cosh r})^{\hf}\partial_u(u-t)^{-is}du\Big|\\
\lesssim_k&(\sinh t)^{-k+\hf}\int_0^{1}(1-v)^{-\hf}v^{-\hf}dv+(\sinh t)^{-k+\hf}\\
&+(\sinh t)^{1-k} \Big[\int_t^{t_0}\partial_u(\frac{u-t}{\cosh u-\cosh r})^{\hf}du-\int_{t_0}^{t+1}\partial_u(\frac{u-t}{\cosh u-\cosh r})^{\hf}du\Big]\\
\lesssim&(\sinh t)^{-k+\hf}+(\sinh t)^{1-k}[2(\frac{u-t}{\cosh u-\cosh r})^{\hf}|_{u=t_0}-(\frac{u-t}{\cosh u-\cosh r})^{\hf}|_{u=t+1}]\\
\lesssim&(\sinh t)^{-k+\hf}.
\end{align*}
\uline{\bm{$B.4(t-r\geq \hf)$}:} 
When $r\geq t-1$, it follows that
\begin{eqnarray*}
B & \les & 
\int_{r}^{t+1}(\sinh u)^{-k}\frac{\sinh u}{(\cosh u-\cosh r)^{\hf}} |u-t|^{-\hf}du
\\
 & \les & 
\int_r^{t+1}(\sinh r)^{-k+\hf}(u-r)^{-\hf}|u-t|^{-\hf}du\lesssim_k (\sinh t)^{-k+\hf}
\end{eqnarray*}
Else, if $1\leq r\leq t-1$, we have
\begin{eqnarray*}
B&\lesssim&\int_{t-1}^{t+1}(\sinh u)^{1-k}(\cosh u-\cosh r)^{-\hf} |u-t|^{-\hf}du\\
&\lesssim&\int_{t-1}^{t+1}(\sinh (t-1))^{-k+\hf}(u-(t-1))^{-\hf}|u-t|^{-\hf}du\\
&\lesssim_k&(\sinh t)^{-k+\hf}.
\end{eqnarray*}

\subsubsection{\bm{$0\leq r<1,r< \hf t$}}\label{rhft}

Let $H(u)=\frac{\sinh u}{(\cosh u-\cosh r)^{\hf}}$,
 we split the integral with respect to $r+1$
\begin{eqnarray}
\eqref{probbb}&\le & 
\Big|\frac{s}{e^{s^2}}\int_{r+1}^{\infty}
H(u)\left(\frac{\partial_u}{\sinh u}\right)^{k-1}\frac{F_z(\be(u-t))-F_z(\be(u+t))}{\sinh u}du
\Big| \nonumber\\
&+&\Big|\frac{s}{e^{s^2}}\int_{r}^{r+1}
H(u)\left(\frac{\partial_u}{\sinh u}\right)^{k-1}\frac{F_z(\be(u-t))-F_z(\be(u+t))}{\sinh u}du
\Big| \nonumber\\
&\triangleq& C+D.
\label{rlhft}
\end{eqnarray}

As in Subsection \ref{evenbig}, it is easier to analyze $C$.
Actually, it is trivially bounded for the cases of $u+t\ge 1$ or $|u-t|\ge 1/2$, when we could apply
 \eqref{bpdeinfty} in  Lemma \ref{lem-2.1}. For the remaining case with $|u-t|\le 1/2$, we have $t\ge 1/2$, $u\sim t$ and
$H(u)\les (\sinh t)^{1/2}$. Then 
  \eqref{bpdezero} in Lemma \ref{lem-2.1}  and \eqref{generalfunc}  gives us
$$
C\les (\sinh t)^{1/2-k}+(\sinh t)^{1/2-k}\int_{(t-\hf)\vee (r+1)}^{t+\hf} 1\vee|u-t|^{-\hf}du\les (\sinh t)^{1/2-k}\ .
$$

Concerning $D$, by the analysis under \eqref{mainpart}, the possible singularity of the integrand near $u=t$ is at most of the order $-1/2$, which is integrable.
As in Subsections \ref{oddsecond} and \ref{oddsmall}, we split the integral depending on the sign of $u-t$: 
\begin{eqnarray}
D&\le&\sum_{l=0}^{k-1}\Big|\frac{s}{e^{s^2}}\int_{u<t,
u\in [r, r+1]
} H(u)\int_{-1}^1\cdots\int_0^{\theta_{l-1}}F_z^{(2l+1)}(\be(t+\theta_l u))d\theta_l\cdots\theta_0 d\theta_0du\Big|\nonumber\\
&&+\Big|\frac{s}{e^{s^2}}\int_{u>t,
u\in [r, r+1]}H(u)\left(\frac{\partial_u}{\sinh u}\right)^{k-1}\frac{F_z(\be(u-t))-F_z(\be(u+t))}{\sinh u}du\Big| 
\label{manyintegral}\\
&\triangleq& D_1+D_2\ .\nonumber
\end{eqnarray}
Notice that owing to Taylor's expansion, we have
$\cosh u-\cosh r\ge (u^2-r^2)/2$ for $u\ge r$, $\sinh u\les u$ for $u\les 1$, and so
\begin{align*}
H(u)=\frac{u}{(u+r)^{\hf}(u-r)^{\hf}}\frac{\sinh u}{u}\frac{(u^2-r^2)^{\hf}}{(\cosh u-\cosh r)^{\hf}}\les\frac{u^{\hf}}{(u-r)^{\hf}}.
\end{align*}

Concerning $D_1$, 
it is trivially bounded when $t\gg 1$, in view of
 \eqref{bpdeinfty} in  Lemma \ref{lem-2.1}.
For the remaining case $t\les 1$,
by \eqref{bpdezero} in Lemma \ref{lem-2.1}, we have
$$F_z^{(2l+1)}(\be(t+\theta_l u))\les e^{2\pi |s|}(1\vee(t+\theta_l u)^{k-5/2-2l})\ .
$$
When $0< t-u\leq 1$ and $
k-5/2-2l<0$, the estimate introduces artificial singularity. To remedy this possible issue, we use integration by parts to get
$$\left|\int_{-1}^{1}\int_0^{\theta_0}\cdots\int_0^{\theta_{l-1}}(t+\theta_l u)^{k-\frac{5}{2}-2l}d\theta_l\cdots \theta_0 d\theta_0\right|
\lesssim u^{-l-1}|(t+\theta u)^{k-3/2-l}|_{\theta=-1}^1|\ ,
$$
for any $l\in [0, k-1]$,
which is controlled by $$\sum_{l\le k-2} t^{k-3/2-l} u^{-l-1}+u^{-k}(t-u)^{-1/2}\les \sum_{l\le k-2} t^{k-5/2-2l}+u^{-k}(t-u)^{-1/2}\les
t^{-k}(t-u)^{-1/2}
 $$
when $t/2\leq u\leq t$ and $u\in  [r,r+1]$.
Then we obtain the desired result for $D_1$ as follows
\begin{eqnarray*}
D_1 & \les &\sum_{l=0}^{k-1} \int_{u<t/2,
u\in [r, r+1]
}\frac{u^{\hf}}{(u-r)^{\hf}}(1\vee t^{k-\frac{5}{2}-2l})du \\
 &  & +\int_{u\in [t/2, t]\cap  [r, r+1]
}  \frac{u^{\hf}}{(u-r)^{\hf}} t^{-k}(t-u)^{-1/2} du
\\
&\lesssim&
\sum_{l=0}^{k-1}(1\vee t^{k-\frac{5}{2}-2l}) \int_{
u\in [r, r+1]}\frac{u^{\hf}}{(u-r)^{\hf}}du 
+t^{1/2-k}\int_{u\in [r, t]
}  \frac{1}{\sqrt{(u-r) (t-u)}}  du \\
& \lesssim& t^{1/2-k}.
\end{eqnarray*}

Turning to $D_2$, for which we have $t\leq r+1$. By the series argument as in 
\eqref{series} and
Subsection \ref{oddsmall}, one can find that the part of $g_z$ is easily bounded by uniform constant. For the part of $h_z$, we keep the initial form and consider only the main item like \two.1.

Let $\nu=\hf(k-\frac{3}{2}-is)$ and $$\left(\shu\right)_l^{k-1}:=\prod\limits_{\al_1+\cdots+\al_k=k-1-l}\partial_u^{\al_j}\left(\shu\right)\cdot u^{2k-l-1}\ .$$
By \eqref{onlyg}, the main items in $D_2$ are of the following form, with 
$l\in [0,k-1]$,
\beq\label{trick}
\begin{split}
& \Big|se^{-s^2}\int_{t }^{r+1}\frac{ u^{1+l-2k} \sinh u}{(\cosh u-\cosh r)^{\hf}} \left(\shu\right)_l^{k-1}(u\pm t)^{k-\frac{3}{2}-l-is}h_z(\be^2(u\pm t)^2) du\Big|\\
=& \Big|e^{-s^2}\int_{t }^{r+1}\frac{(u\pm t)^{\hf}}{(\cosh u-\cosh r)^{\hf}}\frac{\sinh u}{u}(\uf)^{k-1}(\frac{u\pm t}{u})^{k-1-l}\left(\shu\right)_l^{k-1}h_z d(u\pm t)^{-is} \Big|\\
\lesssim&t^{1/2-k}+ \int_{t }^{r+1}\left|\partial_u\left(\frac{(u\pm t)^{\hf}}{(\cosh u-\cosh r)^{\hf}}\frac{\sinh u}{u}(\uf)^{k-1}(\frac{u\pm t}{u})^{k-1-l}\left(\shu\right)_l^{k-1}\right)\right|du,
\end{split}
\eeq
where 
we have used the fact that 
$\left(\shu\right)_l^{k-1}$, $h_z$,
$\partial_u h_z$ are controlled by $e^{\pi |s|}$ by \eqref{sinhr} and Lemma  \ref{thm-series}.
For the last integral, as $\partial_u \frac{\sinh u}{u}$ is bounded, $\partial_u  \left(\shu\right)_l^{k-1}$ is controlled by $u^{-1}$, and $u-r\ge u/2$ as $u>t>2r$,
 we obtain
\beq\label{other}
\begin{split}
D_2\les&t^{1/2-k}+\sum_{\pm}\int_t^{r+1}\frac{(u\pm t)^{\hf}}{(u^2-r^2)^{\hf}}u^{-k}du\\
&+\int_t^{r+1}\left[\frac{(u\pm t)^{-\hf}}{(u^2-r^2)^{\hf}}+\frac{(u\pm t)^{\hf} \sinh u}{(u^2-r^2)^{3/2}}\right] u^{1-k}du\\
&+\sum_{l\le k-2}\int_t^{r+1}\frac{(u\pm t)^{\hf}}{(u^2-r^2)^{\hf}}u^{1-k}(\frac{u\pm t}{u})^{k-2-l}\frac{t}{u^2}du\\
\lesssim&t^{1/2-k}+\int_t^{r+1} (u-t)^{-\hf} u^{-k}du
 \\
 \lesssim&t^{1/2-k}+\int_{t}^{r+1} u^{-\hf-k}du \lesssim t^{-k+\hf}.
\end{split}
\eeq


\subsubsection{\bm{$0\leq r<1,r\geq \hf t$}}\label{bothrtsmall}
In this case, we have $t\le 2r\leq r+1$. The part $C$ as in \eqref{rlhft} could be handled as before. When $\hf t\leq r\leq t$, the first item of \eqref{manyintegral} appears,
and the same analysis gives us
\begin{eqnarray*}
D_1 & \les &
  \int_r^{t\wedge (r+1)
}  \frac{u^{\hf}}{(u-r)^{\hf}} t^{-k}(t-u)^{-1/2} du
\\
&\lesssim&
t^{1/2-k}\int_{u\in [r, t]
}  \frac{1}{\sqrt{(u-r) (t-u)}}  du  \lesssim  t^{1/2-k}.
\end{eqnarray*}

For the remaining part of $D$, i.e., $D_2$, it is similar to that  in Subsection \ref{rhft}, and we will present only the details for the main item of the part involving $h_z$ with $0\leq l\leq k-1$.

\uline{$\hf t\leq r\leq t$, ``$u-t$":} For this case, we have
\begin{align*}
&\Big|se^{-s^2}\int_t^{r+1}\frac{u^{1+l-2k}\sinh u}{(\cosh u-\cosh r)^{\hf}} \left(\shu\right)_l^{k-1}(u-t)^{k-\frac{3}{2}-l-is}h_z(\be^2(u-t)^2)du\Big|\\
=&\Big|e^{-s^2}\int_t^{r+1}\frac{(u-t)^{\hf}}{(\cosh u-\cosh r)^{\hf}}\frac{\sinh u}{u}(\uf)^{k-1}\left(\shu\right)_l^{k-1}(\frac{u-t}{u})^{k-1-l}h_z d (u-t)^{-is}\Big|\\
\lesssim&t^{1/2-k}+\int_t^{r+1}\Big|\partial_u\Big(\frac{(u-t)^{\hf}}{(\cosh u-\cosh r)^{\hf}}\frac{\sinh u}{u}(\uf)^{k-1}\left(\shu\right)_l^{k-1}(\frac{u-t}{u})^{k-1-l}\Big)\Big|du\ .
\end{align*}
With the similar argument in \eqref{other}, we obtain
\beq\label{eq-final}
\begin{split}
D_2^{-}\les& t^{1/2-k}+\int_t^{r+1}\frac{(u-t)^{\hf}}{(u^2-r^2)^{\hf}}u^{-k}du\\
&+\sum_{l\le k-2}\int_t^{r+1}\frac{(u- t)^{\hf}}{(u^2-r^2)^{\hf}}u^{1-k}(\frac{u- t}{u})^{k-2-l}\frac{t}{u^2}du\\
&+\int_t^{r+1}\Big|\frac{\cosh u-\cosh r-(u-t)\sinh u}{(u-t)^{\hf}(\cosh u-\cosh r)^{\frac{3}{2}}}\Big| u^{1-k}du\\
\lesssim&t^{1/2-k}+E\ ,
\end{split}
\eeq
where $$E=\int_t^{r+1}\Big|\frac{\cosh u-\cosh r-(u-t)\sinh u}{(u-t)^{\hf}(\cosh u-\cosh r)^{\frac{3}{2}}}\Big| u^{1-k}du\ .
$$
For the estimate of $E$, we observe 
 from the mean value theorem that, for some $\gamma\in[r,u]$, we have
\begin{align*}
&|(\cosh u-\cosh r)-(u- t)\sinh u|\\
=&|(\cosh u-\cosh r)-(u-r)\sinh u-(r- t)\sinh u|\\
=&|(u-r)\sinh \gamma-(u-r)\sinh u-(r-t)\sinh u|\\
\leq&(u-r)(\sinh u-\sinh r)+(t- r)\sinh u\\
\lesssim&(u-r)^2+u(t- r)\ .
\end{align*}
Then, $E$ could be trivially bounded for the exceptional case $t=r$:
$$E\les  \int_t^{r+1}\Big|\frac{(u-t)^2 }{(u-t)^{2}(u+t)^{3/2}}\Big|u^{1-k}du
\les \int_t^{r+1}u^{-1/2-k}du
\lesssim t^{-k+\hf}\ .$$
On the other hand, for $r<t$, we obtain
\begin{align*}
E\les  &\int_t^{r+1}  \frac{(u-r)^2+u(t-r)}{(u-t)^{\hf}(u+r)^{\frac{3}{2}}(u-r)^{\frac{3}{2}}} u^{1-k}du\\
\les  &\int_t^{r+1}  
(u-t)^{-1/2}u^{-k}+
(t-r)(u-t)^{-1/2}(u-r)^{-3/2}u^{1/2-k} du\\
\lesssim &
1+\int_t^{r+1}  \frac{(u-t)^{1/2}}{u^{k+1}}du+
(t-r) t^{1/2-k}(1+  \int_t^{r+1} 
\frac{(u-t)^{1/2}}{(u-r)^{5/2}}du)
\\\les &
1+\int_t^{r+1}  {u^{-1/2-k}}du+
 t^{1/2-k}(1+  \int_t^{r+1} 
\frac{(t-r)}{(u-r)^{2}}du)
\\
\lesssim&t^{1/2-k }.
\end{align*}

\uline{$\hf t\leq r\leq t$, ``$u+t$":} 
For this case, the worst term is of the form
$$\Big|se^{-s^2}\int_t^{r+1}\frac{u^{1+l-2k}\sinh u}{(\cosh u-\cosh r)^{\hf}} \left(\shu\right)_l^{k-1}(u+t)^{k-\frac{3}{2}-l-is}h_z(\be^2(u+t)^2)du\Big|\ .$$ 
Since the derivative behaves similarly on $\left(\shu\right)_l^{k-1}$ as it does on $(\frac{u+ t}{u})^{k-3/2-l-is}$, we will ignore $\left(\shu\right)_l^{k-1}$ in what follows.
To control the worst term,  we use
integration by parts to control boundary term
\beq\label{trickk}
\begin{split}
&\quad\Big|(\cosh u-\cosh r)^{\hf}(\uf)^{2k-l-1} (u+t)^{k-\frac{3}{2}-l-is}h_z|_{u=t}^{ r+1}\Big|\\
&+\Big|se^{-s^2}\int_t^{r+1}(\cosh u-\cosh r)^{\hf}\partial_u\left[(\uf)^{2k-l-1} (u+t)^{k-\frac{3}{2}-l-is}h_z\right]du\Big|\\
\lesssim&t^{-k+\hf}+\Big|se^{-s^2}\int_t^{r+1}(\cosh u-\cosh r)^{\hf}(\uf)^{2k-l}(u+t)^{k-\frac{3}{2}-l-is}h_zdu\Big|\\
&+\Big|se^{-s^2}\int_t^{r+1}(\cosh u-\cosh r)^{\hf}(\uf)^{2k-l-1}(u+t)^{k-\frac{5}{2}-l-is}h_zdu\Big|\\
\lesssim&t^{-k+\hf}+\Big|e^{-s^2}\int_t^{r+1}(\frac{\cosh u-\cosh r}{u+t})^{\hf}(\uf)^k(\frac{u+t}{u})^{k-l}h_z\partial_u(u+t)^{-is}du\Big|\\
&+\Big|e^{-s^2}\int_t^{r+1}(\frac{\cosh u-\cosh r}{u+t})^{\hf}(\uf)^k(\frac{u+t}{u})^{k-1-l}h_z\partial_u(u+t)^{-is}du\Big|.
\end{split}
\eeq
Similarly, we focus only on the  case when the derivative acts on the first term:
\beq\label{trickkk}
\begin{split}
&\int_t^{r+1}\Big|\frac{(u+t)\sinh u-\cosh u+\cosh r}{(\cosh u-\cosh r)^{\hf}(u+t)^{\frac{3}{2}}}\Big|(\uf)^kdu\\
\lesssim&\int_t^{r+1}\frac{(u+t)^2}{( u-r)^{\hf}( u+r)^{\hf}(u+t)^{\frac{3}{2}}}(\uf)^kdu\\
\lesssim&\int_t^{r+1}( u-r)^{-\hf}u^{-k}du\\
\lesssim&( u-r)^{\hf}u^{-k}|_{u=t}^{r+1}+\int_t^{r+1}( u-r)^{\hf}u^{-k-1}du \lesssim t^{-k+\hf}\ .
\end{split}
\eeq

\uline{$r\geq t$, $u\pm t$:} In this case, the  interval of integration  turns to $\int_r^{r+1}$. By \eqref{trickk} and \eqref{trickkk}, the part of ``$u+t$" is the same after replacing $t$ by  $r$. For the part of $u-t$, it is parallel to 
\eqref{eq-final},
after the integration by parts as in \eqref{trickk}, where the worst term is controlled as follows
\begin{eqnarray*}
E&\les&
\int_r^{r+1}\Big|\pa_u (\frac{\cosh u-\cosh r}{u-t})^{\hf}\Big|(\uf)^kdu\\
&\les&
\int_r^{r+1}\Big|\frac{(u-t)\sinh u-\cosh u+\cosh r}{(\cosh u-\cosh r)^{\hf}(u-t)^{\frac{3}{2}}}\Big|(\uf)^kdu\\
&\lesssim&\int_r^{r+1}\frac{(u-r)^2+u(r-t)}{(u-r)^{\hf}(u+r)^{\hf}(u-t)^{\frac{3}{2}}}(\uf)^kdu\\
&\lesssim&\int_r^{r+1}u^{-1/2-k}du+r^{1/2-k}(r-t)
\int_r^{r+1}\frac{1}{(u-r)^{\hf} (u-t)^{\frac{3}{2}}} du
\\&\les &r^{1/2-k}+r^{1/2-k}\left[1+\int_r^{r+1}\frac{(r-t)(u-r)^{\hf} }{(u-t)^{\frac{5}{2}}} du
\right]\\
&\lesssim&t^{1/2-k}+t^{1/2-k}\int_r^{r+1}\frac{r-t}{(u-t)^{{2}}} du\ 
\les\ t^{1/2-k}\ .
\end{eqnarray*}

\subsection{Consider \eqref{probbbb}}\label{evencos}
The proof of
\eqref{probbbb} follows essentially from the similar proof as that in Subsections \ref{oddcos} and \ref{evensin}, for which we leave the details to the interested readers.

\section{Appendix}\label{sec:appen}
\subsection{Gamma function}
In this part, we recall some fundamental and useful properties of Gamma function to be used, see, e.g., \cite[Appendix A.7]{MR3243734}\cite[(2.5.8)]{MR2683157}.

For real $x$ and $y\neq0$,
\begin{align}
\label{ga}&\ga(x+1+iy)=(x+iy)\ga(x+iy),\\
\label{gaa}&|\ga(x+iy)|\leq |\ga(x)|\ \ ,\ -x\notin\mathbb{N},\\
\label{gam}&\Big|\frac{1}{\ga(x+iy)}\Big|\leq \Big|\frac{1}{\ga(x)}\Big|e^{C(x)|y|^2}\ \ -x\notin\mathbb{N},\\
\label{gammm}&\Big|\frac{1}{\ga(-N+iy)}\Big|\leq |iy||1+iy||2+iy|\cdots |N+iy|e^{|y|^2}\ \ N\in\mathbb{N},\\
\label{gammmm}&\Big|\frac{1}{\ga(x+iy)}\Big|\leq |y|^{\hf-x}e^{\hf \pi |y|}\ \ as\ |y|\to\infty.
\end{align}
where $C(x)=\hf\sum_{k=0}^{\infty}\frac{1}{(k+x)^2}$.
Based on $\eqref{gam}, \eqref{gammm}, \eqref{gammmm}$, we obtain
\begin{equation}\label{gaaa}
\Big|\frac{1}{\ga(x+iy)}\Big|\les_x e^{\pi|y|}\ . 
\end{equation}

\subsection{Modified Bessel functions: proof of \eqref{upperboundofk}}\label{sec:appen.2}
Let $\Re\nu\geq0,x>0$,  by \cite[\S 6.3]{MR1349110}, we have \eqref{eq-Knu}, i.e.,
$$K_{\nu} (x)=\frac{(\hf \pi )^{\hf} x^{\nu} e^{-x}}{\ga(\nu +\hf)} \int_0^{\infty} e^{-x\tau} (\tau+\hf\tau^2)^{\nu-\hf} d\tau .$$
To prove, we split the integral into two parts
\begin{equation*}
K_{\nu}(x)=\frac{(\hf \pi )^{\hf} x^{\nu} e^{-x}(\hf)^{\nu-\hf}}{\ga(\nu +\hf)}(\int_0^1 e^{-x\tau} (2\tau+\tau^2)^{\nu-\hf} d\tau+\int_1^{\infty}e^{-x\tau} (2\tau+\tau^2)^{\nu-\hf} d\tau).
\end{equation*}
The first integral could be treated as follows:
\beq\label{01}
\begin{split}
\Big|\int_0^1e^{-x\tau} (2\tau+\tau^2)^{\nu-\hf} d\tau\Big|&\leq 3^{\Re\nu-\hf}\int_0^1 e^{-x\tau}\tau^{\Re\nu-\hf}d\tau\\
&\lesssim_{\Re\nu}(x^{-\Re\nu-\hf}\int_0^xe^{-t}t^{\Re\nu-\hf}dt\wedge\int_0^1 \tau^{-\hf}d\tau)\\
&\lesssim_{\Re\nu}(x^{-\Re\nu-\hf}\ga(\Re\nu+\hf)\wedge 1)\\
&\lesssim_{\Re\nu}(x^{-\Re\nu-\hf}\wedge 1).
\end{split}
\eeq
Concerning the second integral, we expect 
the main contribution comes from
that of $(\tau^2)^{\nu-\hf}=\tau^{2\nu-1}$.
To illustrate this, we extract the main term as follows:
\beq\label{0infty}
\begin{split}
&\Big|\int_1^{\infty}e^{-x\tau} (2\tau+\tau^2)^{\nu-\hf} d\tau\Big|\\
\leq&\Big|\int_1^{\infty}e^{-x\tau} \tau^{2\nu-1} d\tau\Big|+\Big|\int_1^{\infty}e^{-x\tau}\tau^{\nu-\hf} [(2+\tau)^{\nu-\hf}-\tau^{\nu-\hf}]d\tau\Big|\\
\leq&x^{-2\Re\nu}\Big|\int_x^{\infty}e^{-t} t^{2\nu-1} dt\Big|+C_{\Re\nu}|2\nu-1|x^{-2\Re\nu+1}\int_x^{\infty}e^{-t}t^{2\Re\nu-2}dt\\
:=&A_1+C_{\Re\nu}|2\nu-1|A_2.
\end{split}
\eeq

We first compute the remainder term $A_2$. As $\Re\nu\geq0$, for any $m>1$ 
with $2\Re\nu+m-2\ge 0$,  we have
\beq\label{a2}
\begin{split}
A_2&=x^{-2\Re\nu+1}\int_x^{\infty}t^{-m}e^{-t}t^{2\Re\nu+m-2}dt\\
&\leq C_{\Re\nu,m}x^{-2\Re\nu+1}\int_x^{\infty}t^{-m}dt\\
&=C^1_{\Re\nu,m}x^{-2\Re\nu-m+2}.
\end{split}
\eeq
Similarly, for $A_1$, we get for $m>1$,
\begin{equation}\label{a11}
A_1\leq x^{-2\Re\nu}\int_x^{\infty}t^{-m}e^{-t}t^{2\Re\nu+m-1}dt\leq C^2_{\Re\nu,m} x^{-2\Re\nu-m+1}.
\end{equation}
On the other hand, when $\Re\nu>0$, we have a better estimate for $A_1$:
\begin{equation}\label{a12}
A_1\leq x^{-2\Re\nu}\ga(2\Re\nu)\leq C_{\Re\nu}x^{-2\Re\nu}.
\end{equation}

In summary,  when $x\geq1$, it follows from
\eqref{gaaa}, \eqref{01}, \eqref{a2}, \eqref{a11} with $m=3$ that
$$|K_{\nu}(x)|\leq C_{\Re\nu} e^{\pi|\Im\nu|}x^{-\hf}e^{-x}\ .$$
Else, if $x\leq 1$, by  \eqref{01}, \eqref{a2} with $m=2$,  and \eqref{a12}, we have
\begin{equation*}
|K_{\nu}(x)|\leq 
C_{\Re\nu}e^{\pi|\Im\nu|} x^{-\Re\nu}\ ,
\end{equation*}
if $\Re \nu>0$.

The previous estimate for $\Re\nu=0$ near $x=0$ is not good, due to the 
integrability requirement $m>1$ from \eqref{a11}. Instead, when $x<1$ and $\Re\nu=0$ (i.e. $\nu=is$), we can deal with $A_1$ in the following way:
\begin{equation*}
\Big|\int_x^{\infty}e^{-t}t^{2is-1}dt\Big|\leq \int_x^1+\int_1^{\infty}e^{-t}t^{-1}dt\lesssim |\ln x|+1\ .
\end{equation*}
Alternatively, when
 $s\neq 0$, we can make use of  the power  $-1+2is$
\begin{equation*}
\Big|\int_x^{\infty}e^{-t}t^{2is-1}dt\Big|=\Big|\frac{1}{2is}\int_x^{\infty}e^{-t}dt^{2is}\Big|\leq \Big|\frac{1}{2is}e^{-x}\Big|+\Big|\frac{1}{2is}\int_x^{\infty}e^{-t}t^{2is}dt\Big|\lesssim |s|^{-1}.
\end{equation*}
Together with
 \eqref{01}, \eqref{a2} with $m=2$, we   obtain the following bound 
$$|K_{is}(x)|\leq C e^{\pi|s|}\left[\left(\frac{1}{|s|}\wedge\ln\frac{1}{x}\right)+1\right]\ , \forall x\in (0,1).$$
which completes the proof.


\begin{thebibliography}{10}

\bibitem{MR2902129}
Jean-Philippe Anker, Vittoria Pierfelice, and Maria Vallarino.
\newblock The wave equation on hyperbolic spaces.
\newblock {\em J. Differential Equations}, 252(10):5613--5661, 2012.

\bibitem{MR3345662}
Jean-Philippe Anker, Vittoria Pierfelice, and Maria Vallarino.
\newblock The wave equation on {D}amek-{R}icci spaces.
\newblock {\em Ann. Mat. Pura Appl. (4)}, 194(3):731--758, 2015.

\bibitem{MR143935}
N.~Aronszajn and K.~T. Smith.
\newblock Theory of {B}essel potentials. {I}.
\newblock {\em Ann. Inst. Fourier (Grenoble)}, 11:385--475, 1961.

\bibitem{MR2683157}
Richard Beals and Roderick Wong.
\newblock {\em Special functions}, volume 126 of {\em Cambridge Studies in
  Advanced Mathematics}.
\newblock Cambridge University Press, Cambridge, 2010.
\newblock A graduate text.

\bibitem{GLS97}
Vladimir Georgiev, Hans Lindblad, and Christopher~D. Sogge.
\newblock Weighted {S}trichartz estimates and global existence for semilinear
  wave equations.
\newblock {\em Amer. J. Math.}, 119(6):1291--1319, 1997.

\bibitem{MR3243734}
Loukas Grafakos.
\newblock {\em Classical {F}ourier analysis}, volume 249 of {\em Graduate Texts
  in Mathematics}.
\newblock Springer, New York, third edition, 2014.

\bibitem{MR4026182}
Yannick Sire, Christopher~D. Sogge, and Chengbo Wang.
\newblock The {S}trauss conjecture on negatively curved backgrounds.
\newblock {\em Discrete Contin. Dyn. Syst.}, 39(12):7081--7099, 2019.

\bibitem{MR4169670}
Yannick Sire, Christopher~D. Sogge, Chengbo Wang, and Junyong Zhang.
\newblock Strichartz estimates and {S}trauss conjecture on non-trapping
  asymptotically hyperbolic manifolds.
\newblock {\em Trans. Amer. Math. Soc.}, 373(11):7639--7668, 2020.

\bibitem{MR0290095}
Elias~M. Stein.
\newblock {\em Singular integrals and differentiability properties of
  functions}.
\newblock Princeton Mathematical Series, No. 30. Princeton University Press,
  Princeton, N.J., 1970.

\bibitem{MR1804518}
Daniel Tataru.
\newblock Strichartz estimates in the hyperbolic space and global existence for
  the semilinear wave equation.
\newblock {\em Trans. Amer. Math. Soc.}, 353(2):795--807, 2001.

\bibitem{MR1395148}
Michael~E. Taylor.
\newblock {\em Partial differential equations. {I}}, volume 115 of {\em Applied
  Mathematical Sciences}.
\newblock Springer-Verlag, New York, 1996.
\newblock Basic theory.

\bibitem{MR2743652}
Michael~E. Taylor.
\newblock {\em Partial differential equations {II}. {Q}ualitative studies of
  linear equations}, volume 116 of {\em Applied Mathematical Sciences}.
\newblock Springer, New York, second edition, 2011.

\bibitem{MR1349110}
G.~N. Watson.
\newblock {\em A treatise on the theory of {B}essel functions}.
\newblock Cambridge Mathematical Library. Cambridge University Press,
  Cambridge, 1995.
\newblock Reprint of the second (1944) edition.

\end{thebibliography}
\end{document}